\documentclass[a4paper,11pt]{article}
\usepackage[a4paper]{geometry}
\usepackage{graphicx}
\usepackage{amssymb,amsmath, amsthm}
\usepackage{mathrsfs}
\usepackage{bbm,bbold}
\usepackage{lineno}
\usepackage{enumitem}

\newtheorem{theorem}{Theorem}
\newtheorem{lemma}[theorem]{Lemma}
\newtheorem{proposition}[theorem]{Proposition}

\newcommand{\myboldmath}[1]{\mathbbm{#1}}
\newcommand{\R}{\myboldmath{R}}
\newcommand{\G}{\myboldmath{G}}
\newcommand{\Q}{\myboldmath{Q}}

\newcommand{\Z}{\myboldmath{Z}}

\newcommand{\1}{\mathbf{1}}

\newcommand{\im}{\operatorname{im}}
\DeclareMathOperator{\supp}{supp}
\newcommand{\up}{\textup{up}}
\newcommand{\down}{\textup{down}}

\begin{document}

\title{Higher Dimensional Discrete Cheeger Inequalities}
\author{Anna Gundert\footnote{Universit\"at zu K\"oln, Weyertal 86-90, D-50923 K\"oln, Germany. \texttt{anna.gundert@uni-koeln.de}. Research supported by the Swiss National Science Foundation (SNF Projects 200021-125309 and 200020-138230).} \and May Szedl\'ak\footnote{Institut f\"{u}r Theoretische Informatik, ETH Z\"{u}rich, CH-8092 Z\"{u}rich, Switzerland. \texttt{may.szedlak@inf.ethz.ch}.}}

\maketitle

\begin{abstract}
For graphs there exists a strong connection between spectral and combinatorial expansion properties. This is expressed, e.g., by the discrete Cheeger inequality, the lower bound of which states that $\lambda(G) \leq h(G)$, where $\lambda(G)$ is the second smallest eigenvalue of the Laplacian of a graph $G$ and $h(G)$ is the Cheeger constant measuring the edge expansion of~$G$. We are interested in generalizations of expansion properties to finite simplicial complexes of higher dimension (or uniform hypergraphs).

Whereas higher dimensional Laplacians were introduced already in 1945 by Eckmann, the generalization of edge expansion to simplicial complexes  is not straightforward. 
Recently, a topologically motivated notion analogous to edge expansion that is based on $\mathbb{Z}_2$-cohomology was introduced by Gromov and independently by Linial, Meshulam and Wallach. It is known that for this generalization there is no direct higher dimensional analogue of the lower bound of the Cheeger inequality.

A different, combinatorially motivated generalization of the Cheeger constant, denoted by $h(X)$, was studied by Parzanchevski, Rosenthal and Tessler. They showed that indeed $\lambda(X) \leq h(X)$, where $\lambda(X)$ is the smallest non-trivial eigenvalue of the ($(k-1)$-dimensional upper) Laplacian, for the case of $k$-dimensional simplicial complexes $X$ with complete $(k-1)$-skeleton.

Whether this inequality also holds for $k$-dimensional complexes with non-com\-plete $(k-1)$-skeleton has been an open question.
We give two proofs of the inequality for arbitrary complexes. The proofs differ strongly in the methods and structures employed,
and each allows for a different kind of additional strengthening of the original result.
\end{abstract}

\newpage
\section*{Introduction}
Roughly speaking, a graph is an expander if it is sparse and at the same time well-connected. Such graphs have found various applications, in theoretical computer science as well as in pure mathematics.
Expander graphs have, e.g., been used to construct certain classes of error correcting codes, in a proof of the PCP Theorem (\cite{Dinur:2007}, see also \cite{RadhakrishnanSudan:2007}), a deep result in computational complexity theory, and in the theory of metric embeddings.
See, e.g., the surveys \cite{HooryLinialWigderson:2006} and \cite{Lubotzky:2011} for these and other applications.

In recent years, the combinatorial study of simplicial complexes - considering them as a higher-dimensional generalization of graphs - has attracted increasing attention and the profitability of the concept of expansion for graphs has inspired the search for a corresponding higher-dimensional notion, see, e.g., \cite{GundertWagner:2012, Lubotzky:2013, Parzanchevski:2012, Steenbergen:2012}.

The expansion of a graph $G$ can be measured by the \emph{Cheeger constant}\footnote{\label{foot1}Often the Cheeger constant is defined in a slightly different but closely related way, see Section~\ref{sec:DiffNotions}.}
\[
h(G) := \min_{{A \subseteq V} \atop {0 <|A| <|V|}}\frac{|V||E(A,V\setminus A)|}{|A||V\setminus A|}.
\]
Here $E(A,V\setminus A)$ is the set of edges with one endpoint in $A$ and the other in $V \setminus A$.
A straightforward higher-dimensional analogue is the following \emph{Cheeger constant} of a $k$-dimensional simplicial complex $X$ with \emph{complete} $(k-1)$-skeleton, studied in \cite{Parzanchevski:2012}:
\[
h(X) := \min_{{V=\amalg_{i=0}^k A_i} \atop {A_i \neq \emptyset}} \frac{|V||F(A_0,A_1,\dots,A_k)|}{|A_0|\cdot|A_1|\cdot\ldots\cdot|A_k|}.
\]
Here $F(A_0,A_1,\dots,A_k)$ is the set of $k$-dimensional faces of $X$ with exactly one vertex in each set $A_i$.
A different, more topologically motivated notion based on $\mathbb{Z}_2$-cohomology was introduced by Gromov \cite{Gromov:2010} and independently by Linial, Meshulam and Wallach \cite{LinialMeshulam:2006, MeshulamWallach:2009}. We will not work with this notion but we later consider an adapted version of $h(X)$ inspired by it. See discussion of results and Section \ref{sec:DiffNotions} for more details.

For graphs, this combinatorial notion of expansion is connected to the spectra of certain matrices associated with the graph: the adjacency matrix and the Laplacian.
This connection between combinatorial and spectral expansion properties of a graph is established, e.g., by the discrete Cheeger inequality \cite{Alon:1986,Alon:1985,Dodziuk:1984, Tanner:1984}. For a graph $G$ with second smallest eigenvalue $\lambda(G)$ of the Laplacian $L(G)$ (see Section~\ref{sec:Prelims}) and maximum degree $d_\text{max}$, it states that
\[
\lambda(G) \leq h(G) \leq \sqrt{8d_\text{max}\lambda(G)}.
\]

A different approach to generalizing expansion is hence to consider higher-di\-men\-sion\-al analogues of graph Laplacians.
Higher-dimensional Laplacians were first introduced by Eckmann \cite{Eckmann:1945} in the 1940s and have since then been used in various contexts, see \cite{Kalai:1983} for an example.
We denote by $\lambda(X)$ the smallest non-trivial eigenvalue of this Laplacian. More precisely, $\lambda(X)$ is the smallest eigenvalue of the upper Laplacian $L^\up_{k-1}(X)$ on $(B^{k-1}(X;\R))^\perp$. (See Section~\ref{sec:Prelims} for further details.)

The Cheeger inequality for graphs has proven to be a useful tool. Computing the Cheeger constant is difficult, from the standpoint of complexity theory \cite{Matula:1990, Blum:1981} but often also for explicit examples.
The lower bound -- even though easy to prove -- hence gives a helpful, polynomially computable, lower bound on the Cheeger constant.
Many constructions of families of expander graphs (graphs $(G_n)_{ n \in \mathbb{N}}$ on $n$ vertices with constant edge degree, where the Cheeger constant is bounded from below by a constant) use eigenvalues to establish a lower bound on the combinatorial expansion \cite{Gabber:1981, Lubotzky:1988, Margulis:1973,Margulis:1988, ReingoldVadhanWigderson:2002}.

Parzanchevski, Rosenthal and Tessler \cite{Parzanchevski:2012} recently showed the following analogue of this lower bound of the Cheeger inequality for $k$-dimensional simplicial complexes with complete $(k-1)$-skeleton.
\begin{theorem}[Parzanchevski et al. \cite{Parzanchevski:2012}]\label{thm:Parz}
Let $X$ be a $k$-dimensional simplicial complex with complete $(k-1)$-skeleton. Then 
$\lambda(X) \leq h(X)$.
\end{theorem}

Here, we present two ways to extend this result to \mbox{$k$-d}i\-men\-sion\-al complexes with non-complete $(k-1)$-skeleton, addressing an open question that was posed in \cite{Parzanchevski:2012}. Both proofs allow for an additional strengthening of the original result.

\paragraph{} To make an extension to arbitrary complexes possible, it is necessary to adapt the definition of $h(X)$, as it is easily seen that $h(X)$ as defined above is non-zero only for $k$-dimensional $X$ with complete $(k-1)$-skeleton.
For any \mbox{$k$-d}i\-men\-sion\-al complex $X$, define its \emph{$k$-dimensional completion} as $K(X):=X \cup \{\tau^{\partial} \in \binom{V}{k+1} \colon \tau^{\partial}\setminus\{v\} \in X \text{ for all } v \in \tau^{\partial}\}$. If $X$ has a complete $(k-1)$-skeleton, we get $K(X) = K_n^k$, the complete $k$-dimensional complex on $n$ vertices.
We then define, as suggested in \cite{Parzanchevski:2012},
\[
h(X) := \min_{{V=\amalg_{i=0}^k A_i} \atop {A_i \neq \emptyset}} \frac{|V||F(A_0,A_1,\dots,A_k)|}{|F^{\partial}(A_0,A_1,\dots,A_k)|},
\]
where $F^{\partial}(A_0,A_1,\dots,A_k)$ is defined as
\[
\{\tau^{\partial} \in \tbinom{V}{k+1} \colon \tau^{\partial} \in K(X), |\tau^{\partial} \cap A_i| = 1 \text{ for } i=0,1,\ldots,k\}.
\]
Note that this is the set corresponding to $F(A_0,A_1,\dots,A_k)$ in the completion $K(X)$ -- and hence the largest possible set of $k$-simplices with one vertex in each $A_i$ in a simplicial complex with the $(k-1)$-skeleton of $X$.
For a partition with \mbox{$|F^{\partial}(A_0,A_1,\dots,A_k)| = 0$} define $\frac{|V||F(A_0,A_1,\dots,A_k)|}{|F^{\partial}(A_0,A_1,\dots,A_k)|} := \infty$.
Our first result is as follows:
\begin{theorem} \label{thm:sparse}
Let $X$ be a $k$-dimensional simplicial complex and let $A_0,A_1,\dots,A_k$ be a partition of $V$ that minimizes $h(X)$. 
Let
\[
C(X) := \max\biggl\{\sum_{v \in \tau^{\partial}} d(\tau^{\partial}\setminus\{v\}): \tau^{\partial} \in F^{\partial}(A_0,A_1,\dots,A_k)\biggr\},
\]
where $
d(\sigma) := |\{\tau^{\partial} \in F^{\partial}(A_0,A_1,\dots,A_k): \sigma \subseteq \tau^{\partial}\}|
$ for a $(k-1)$-face $\sigma \in X_{k-1}$.
Then \[\lambda(X) \leq \frac{C(X)}{|V|}h(X).\]
In particular, if every $(k-1)$-face  is contained in at most $C$ $k$-faces of $K(X)$, i.e., 
\[\forall \sigma \in X_{k-1}: |\{\tau \in K(X) :\sigma \subseteq \tau \}| \leq C,\]
then $C(X) \leq (k+1) C$ and hence 
\[\lambda(X) \leq \frac{(k+1)C}{|V|}\cdot  h(X).\]
\end{theorem}
The latter inequality is useful, since in general $C(X)$ is hard to compute.
Note that $d(\sigma)$ as well as $C(X)$ depend on the minimizing partition $A_0,A_1,\dots,A_k$, which is not necessarily unique. However, for better readability we omit the partition from the notation.

Observe that if $A_l$ is the unique block not containing a vertex of $\sigma$, then $d(\sigma) \leq |A_l|$ and this bound is tight for $k$-complexes $X$ with complete $(k-1)$-skeleton. So by definition $C(X) \leq |V|$ and Theorem~\ref{thm:sparse} implies the statement of Theorem~\ref{thm:Parz} for arbitrary $k$-dimensional simplicial complexes. While $C(X)=|V|$ for $X$ with complete $(k-1)$-skeleton, in extreme cases $C(X)$ can be arbitrary small compared to~$|V|$.  

\paragraph{} Our second result strengthens Theorem~\ref{thm:Parz} in a different way. It is possible to rephrase $h(X)$ in terms of $\Z_2$-coboundaries as follows (see Section~\ref{sec:Prelims} for the necessary definitions):
For a partition $V=\amalg_{i=0}^k A_i$ let $F(A_0,A_1,\ldots,A_{k-1})$ be the set of $(k-1)$-dimensional faces of $X$ with exactly one vertex in each set $A_i$, $i=0,1,\ldots,k-1$. Let $\1_{F(A_0,A_1,\ldots,A_{k-1})}$ be its characteristic function, interpreted as a $\Z_2$-cochain.
Then the support of the $\Z_2$-coboundary $\delta_X \1_{F(A_0,A_1,\ldots,A_{k-1})}$ in $X$, the set of $k$-simplices of $X$ with an odd number of $(k-1)$-faces in $F(A_0,A_1,\ldots,A_{k-1})$, is exactly the set $F(A_0,A_1,\ldots,A_k)$. The corresponding coboundary $\delta_{K(X)} \1_{F(A_0,A_1,\ldots,A_{k-1})}$ in $K(X)$ has support $F^{\partial}(A_0,A_1,\ldots,A_k)$. Thus,

\[
h(X) = \min_{{V=\amalg_{i=0}^k A_i} \atop {A_i \neq \emptyset}}\frac{|V| \cdot |\delta_X \1_{F(A_0,A_1,\ldots,A_{k-1})}|}{|\delta_{K(X)} \1_{F(A_0,A_1,\ldots,A_{k-1})}|},
\]
where $|\cdot|$ denotes the Hamming norm.
In order to strengthen the bound on $\lambda(X)$ given by Theorem~\ref{thm:Parz}, we define
\[
h'(X) := \min_{\substack{V=\amalg_{i=0}^{k-1} A_i,\, f \in C^{k-1}(X,\Z_2),\\\supp(f)\subset F(A_0,A_1,\ldots,A_{k-1})}} \frac{|V| \cdot |\delta_X f|}{|\delta_{K(X)} f|}.
\]
If $|\delta_{K(X)} f| = 0$, we again define $\frac{|V| \cdot |\delta_X f|}{|\delta_{K(X)} f|}=\infty$.
As an illustration think of the case of dimension $k=2$. Then the cochains $f$ considered here are those whose support describes the edge set of a bipartite graph. We compare the number of triangles in $X$ with an odd number of egdes, i.e., exactly one edge, in the support of $f$ with the number of possible such triangles.

Note that here we consider partitions of $V$ into $k$ parts. For a partition $V=\amalg_{i=0}^k A_i$ into $k+1$ parts, clearly 
\[
\supp(\1_{F(A_0,A_1,\ldots,A_{k-1})})\subset F(A_0,A_1,\ldots,A_{k-2},A_{k-1}\cup A_k).
\]
Hence, as we minimize over a larger set of cochains, we have $h'(X) \leq h(X)$. See below for an example where $h'(X) < h(X)$. We show:
\begin{theorem}\label{thm:Theorem2}
Let $X$ be a $k$-dimensional simplicial complex.
Then $\lambda(X) \leq h'(X).$
\end{theorem}


\subsection*{Discussion of Results.}
\begin{enumerate}

\item The inspiration for the definition of $h'(X)$ is a different analogue of the Cheeger constant for graphs, introduced by Gromov and independently by Linial, Meshulam and Wallach. 
It is based on $\mathbb{Z}_2$-cohomology and emerged in various contexts as a useful notion, see, e.g., \cite{Gromov:2010,LinialMeshulam:2006, MeshulamWallach:2009}.
For a $k$-complex with complete $(k-1)$-skeleton, this notion can be described\footnote{\label{foot2}Usually, one considers a different notion. For $k$-complexes with complete $(k-1)$-skeleton, the two notions are closely related, see Section~\ref{sec:DiffNotions} for more details.}
by 
\[
\min_{f \in C^{k-1}(X,\Z_2)} \frac{|V| \cdot |\delta_X f|}{|\delta_{K_n^k} f|},
\]
 similar to the definitions of $h(X)$ and $h'(X)$, but without any restriction on the cochains considered.
As this seems to be an important and useful concept, one might wish for an inequality as in Theorems~\ref{thm:Parz}, \ref{thm:sparse} and \ref{thm:Theorem2} also for this notion of expansion. It was, however, shown that such an inequality can not exist, see \cite{GundertWagner:2012,Steenbergen:2012}.

So far -- except for a technical explanation, see Lemma~\ref{lem:connectionL - Z2} -- we have no deeper insight into why the Cheeger inequality does hold when the definition of expansion is restricted to the special class of cochains appearing in the definition of~$h'(X)$.

\item Theorem~\ref{thm:sparse} and Theorem~\ref{thm:Theorem2} can indeed give a stronger bound than Theorem~\ref{thm:Parz}, see Section~\ref{sec:Examples} for examples that also show that it depends on the complex $X$ whether $h'(X)$ or $\frac{C(X)}{|V|}h(X)$ presents the stronger upper bound on $\lambda(X)$.

\item Recall that the Cheeger inequality for graphs also gives an upper bound of $h(G)$ in terms of $\lambda(G)$. As $\lambda(X) = 0$ does not imply $h(X) =0$, see  \cite{Parzanchevski:2012}, a higher-dimensional analogue of this upper bound of the form $C \cdot h(X)^m \leq \lambda(X)$ is hence not possible.
\end{enumerate}

\section{Preliminaries}\label{sec:Prelims}


\paragraph{Graph Laplacian.}
Let $G=(V,E)$ be a finite simple undirected graph, $|V| =n$.
The \emph{Laplacian} of $G$ is the $|V| \times |V|$-matrix  
\[
L(G) = D(G) - A(G).
\]
Here, $A(G)$ is the \emph{adjacency matrix} -- given by $A_{u,v}=1$ if and only if $\{u,v\}\in E$ -- and $D(G)$ is the diagonal matrix with entries $D_{v,v}=\deg_G(v)$, the \emph{degrees} of the vertices.

The Laplacian is a symmetric positive semi-definite matrix and hence has $n$ real non-negative eigenvalues. As $L\1 = 0$, the smallest eigenvalue is always $0$, and we denote by $\lambda(G)$ the second smallest eigenvalue of $L(G)$. A graph $G$ is connected if and only if $\lambda(G)$ is non-zero (see, e.g., \cite{HooryLinialWigderson:2006}).





\paragraph{Simplicial Complexes.}
Let $V$ be a finite set. A \emph{(finite abstract) simplicial complex} (or \emph{complex}) $X$ with vertex set $V$ is a collection of subsets of $V$ that is closed under taking subsets, i.e., $\sigma \subseteq \tau \in X$ implies $\sigma \in X$.

An element $\tau \in  X$ is called a \emph{simplex} or \emph{face} of $X$, the \emph{dimension} of $\tau$ is $\dim\tau=|\tau|-1$. A simplex $\tau$ with $\dim \tau =i$ is also called an \emph{$i$-simplex}. The \emph{dimension} of the complex $X$ is $\dim X = \max_{\tau \in X} \dim \tau$. A simplicial complex of dimension $k$ is called a \emph{$k$-dimensional simplicial complex} or a \emph{$k$-complex}.
The one-element sets $\{v\}$, $v \in V$, are the \emph{vertices} of $X$. We identify the singleton $\{v\}$ with its unique element~$v$.
For an $(i-1)$-simplex $\sigma$ the \emph{degree} of $\sigma$ is defined as $\deg \sigma = |\{\tau \supseteq \sigma | \dim \tau = i\}|$.
The set of all $i$-simplices of $X$ is denoted by $X_i$, the collection of all simplices of dimension at most $i$, the \emph{$i$-skeleton} of $X$, by $X^{(i)}$. 
The \emph{complete $k$-complex} $K_n^k$ has vertex set $V = [n] = \{1,\ldots,n\}$ and $X_i=\binom{[n]}{i+1}$ for all $i \leq k$. 

\paragraph{Cohomology.}
Let $X$ be a $k$-dimensional simplicial complex with vertex set $V$ and assume that we have a fixed linear ordering on $V$. We consider the faces of $X$ with the \emph{orientation} given by the order of their vertices.

Formally, consider an $i$-simplex $\tau = \{v_0,v_1,\dots,v_i\}$ where $v_0 < v_1 < \dots <v_i$. For an $(i-1)$-simplex $\sigma \in X_{i-1}$ the \emph{oriented incidence number} $[\tau:\sigma]$ is defined as $(-1)^j$ if $\sigma = \tau \setminus \{v_j\}$, for some $j = 0,1,\dots,i$ and zero otherwise, i.e., if $\sigma \nsubseteq \tau$.
In particular for $v \in X_0$ and the unique face $\emptyset \in X_{-1}$ we have $[v:\emptyset] = 1$.

Let $\G$ be an Abelian group (we will be concerned with the cases $\G = \mathbb{Z}_2$ and $\G = \R$).
The \emph{group of $i$-dimensional cochains} on $X$ (with coefficients in $\G$) is 
\[
C^i(X,\G): = \{ f: X_i \to \G\},
\]
i.e., the group of maps from the set of $i$-simplices to $\G$.
For $i < -1$ or $i > \dim X$ we conveniently define $C^i(X,\G) =0$.
Note that since the empty set is the unique element of $X_{-1}$ we have $C^{-1}(X,\G) \cong \G$. 
The characteristic functions $e_{\tau}$ of faces $\tau \in X_i$ form a basis of $C^i(X, \G)$, they are called the \emph{elementary cochains}.

The \emph{coboundary operator} $\delta_{i} \colon C^{i}(X,\G) \rightarrow C^{i+1}(X,\G)$ is the linear function given by
\[ \delta_{i}f (\tau):= \sum_{\sigma \in X_{i}} [\tau:\sigma]f(\sigma),\]
for $\tau$ an $(i+1)$-simplex, $f \in C^{i}(X,\G)$ and $-1\leq i < \dim X$. We let $\delta_i =0$ otherwise.
Define $Z^i(X;\G) := \ker \delta_i$  the group of $i$-dimensional \emph{cocycles} and
$B^i(X;\G) := \im \delta_{i-1}$ the group of $i$-dimensional \emph{coboundaries}.
A straightforward calculation shows that \mbox{$\delta_i \delta _{i-1}=0$}, i.e., $B^i(X;\G) \subseteq Z^i(X;\G)$.
Hence, we can define the (reduced) \emph{$i$-th cohomology group with coefficients in $\G$} as 
\[
\tilde{H}^i(X;\G) := Z^i(X;\G)/ B^i(X;\G).
\]

\paragraph{Real Coefficients and Higher-Dimensional Laplacians.}
We endow $C^i(X;\R)$ with the \emph{inner product} 
\[
\langle f, g \rangle = \sum_{\tau \in X_i}f(\tau)g(\tau)
\]
for $f,g \in C^i(X;\R)$ and denote the dual operator of $\delta_{i-1}$ by $\partial_i\colon C^i(X;\R) \rightarrow C^{i-1}(X;\R)$, i.e., for $f \in C^i(X;\R)$ and $g \in C^{i-1}(X;\R)$ we have 
\[
\langle \partial_i f, g \rangle = \langle f, \delta_{i-1} g \rangle.
\]
The map $\partial_i$ is also called the \emph{boundary operator} and the groups \mbox{$Z_i(X;\R) := \ker \partial_i$} and \mbox{$B_i(X;\R) := \im \partial_{i+1}$} are the group of $i$-di\-men\-sion\-al \emph{cycles} and the group of $i$-di\-men\-sion\-al \emph{boundaries}, respectively.

Setting $\mathcal{H}_i=\mathcal{H}_i(X;\R):= Z_i(X;\R) \cap Z^i(X;\R)$, one gets a \emph{Hodge decomposition} of the vector space $C^i(X;\R)$ into pairwise orthogonal subspaces
\begin{equation}\label{eqn:hodge-decomp}
 C^i(X;\R)=\mathcal{H}_i\oplus B^i(X;\R)\oplus B_i(X;\R),
\end{equation}
in particular, $\mathcal{H}_i\cong H^i(X;\R)$ (see \cite{Eckmann:1945,HorakJost}).

The higher-dimensional analogue of the graph Laplacian is based on these notions.
From now on, write $C^i$ for $C^i(X;\R)$, $B^i$ for $B^i(X;\R)$ and $Z_i$ for $Z_i(X;\R)$.
The \emph{upper, lower} and \emph{full Laplacian}
$L_i^{\text{up}}(X), L_i^{\text{down}}(X), L_i(X) \colon C^{i} \to C^{i}$
in dimension $i$ are defined as 
\[
L_{i}^{\text{up}}(X) := \partial_{i+1} \delta_{i}, \quad
L_{i}^{\text{down}}(X) := \delta_{i-1} \partial_{i} \quad \text{and}\quad
L_{i}(X) := L_{i}^{\text{up}}(X) + L_{i}^{\text{down}}(X),
\]
respectively.
We solely focus on the case $i=k-1$.

Analogously to the case of graphs ($k=1$) we can express $L^\up_{k-1}(X)$ as a matrix: With respect to the orthogonal basis of elementary cochains it corresponds to the matrix 
\[
L^\up_{k-1}(X) = D_{k-1}(X) - A_{k-1}(X).
\]
Here we let $D_{k-1}(X)$ denote the diagonal matrix with entry $(D_{k-1}(X))_{\tau,\tau}=\deg(\tau)$ for $\tau \in X_{k-1}$ and define the 
\emph{adjacency matrix} $A_{k-1}(X)$ by 
\[
 (A_{k-1}(X))_{\tau,\tau'} = \begin{cases}
                -[\tau \cup \tau':\tau][\tau \cup \tau':\tau']& \text{if $\tau \sim \tau'$},\\
		        0& \text{otherwise,}
               \end{cases}
\]
where $\tau,\tau' \in X_{k-1}$ and we write $\tau \sim \tau'$ if $\tau$ and $\tau'$ share a common $(k-2)$-face and $\tau \cup \tau' \in X_{k}$. Note that for a graph $G$, we get $-[\{v,v'\}\!:\!v][\{v,v'\}\!:\!v']=1$ for all non-zero entries of $A_0(G)$.
This shows that $A_0(G)=A(G)$ and that $L^\up_{0}(G)$ agrees with the Laplacian $L(G)$.
For $k>1$, non-zero entries of $A_{k-1}(X)$ do not necessarily have the same sign.

Note that $L_{k-1}^{\text{up}}(X)$ (as well as $L_{k-1}^{\text{down}}(X)$ and $L_{k-1}(X)$) is a \emph{self-adjoint} and \emph{positive semidefinite} linear operator. 
It is furthermore not hard to see that $\ker L_{k-1}^{\text{up}}(X)=Z^{k-1}$. As $B^{k-1}\!\!\subseteq\!\!Z^{k-1}$, this implies that $L_{k-1}^{\text{up}}(X)$ is zero on $B^{k-1}$. Hence, non-zero eigenvalues can only occur in $(B^{k-1})^{\bot}$ and we define the \emph{spectral gap} of $X$ as 
\begin{linenomath}
\begin{equation*}
\lambda(X) := \text{minSpec}(L_{k-1}^\up(X)|_{(B^{k-1})^{\bot}}) 
=\text{minSpec}(L_{k-1}^\up(X)|_{Z_{k-1}}),
\end{equation*}
\end{linenomath}

where the equality holds because  we have $Z_{k-1}=(B^{k-1})^{\bot}$ by the Hodge decomposition \eqref{eqn:hodge-decomp}.
We remark that even though the spaces $B^{k-1}$ and $Z_{k-1}$ depend on the choice of orientations for the faces of $X$, the spectrum of $L_{k-1}^\up$ and the value of $\lambda(X)$ do not.

Note that $\lambda(X)$ is also the minimal eigenvalue of the full Laplacian $L_{k-1}(X)$ on $Z_{k-1}$, since $Z_{k-1} = \ker L_{k-1}^{\text{down}}(X)$.
We have $\lambda(X) =0$, i.e., there exist more zero eigenvalues than the ones corresponding to functions in $B^{k-1}$, if and only if \mbox{$H^{k-1}(X;\R) \neq 0$}.

For a graph $G$ the space $B^{0}$ is the space of constant functions, spanned by the all-ones vector $\1$, so this definition of the spectral gap coincides with $\lambda(G)$ as defined previously.

\section{Different Notions of Higher-Dimensional Expansion}\label{sec:DiffNotions}
The notions of expansion considered here are based on the Cheeger constant for graphs defined as
$
h(G) := \min_{{A \subseteq V} \atop {0 <|A| <|V|}}\frac{|V||E(A,V\setminus A)|}{|A||V\setminus A|}.
$
Often the Cheeger constant is defined differently by
\[
\phi (G) = \min_{A \subseteq V,0 < |A| \leq |V|/2}\frac{|E(A,V\setminus A)|}{|A|}.
\]
Since $\phi(G) \!\leq \!h(G)\! \leq \!2\phi(G)$, these two concepts are closely related. Both are also called \emph{(edge) expansion ratio}.

An analogue of the latter graph parameter for higher dimensional simplicial complexes was introduced by Gromov and independently by Linial, Meshulam and Wallach.
For a $k$-dimensional complex $X$ define
\[
\phi(X):=\min_{f \in C^{k-1}(X,\Z_2)} \frac{|\delta_X f|}{|[f]|}
,\]
 where $|[f]| = \min\{|f+\delta_X g|\colon g\in C^{k-2}(X;\Z_2)\}$, see, e.g., \cite{Gromov:2010,LinialMeshulam:2006, MeshulamWallach:2009}. Here, $\delta_X$ denotes the $\Z_2$-coboundary operator of the complex $X$.
 For $f$ with $|[f]| =0$, define  $\frac{|\delta_X f|}{|[f]|} = \infty$.
Note that $\phi(X) = 0$ if and only if the complex $X$ has a non-trivial $\Z_2$-cocycle, i.e., iff $\tilde{H}^{k-1}(X;\Z_2)\neq 0$.
 
Our definition of $h'(X)$ is inspired by the following parameter, generalizing the definition of $h(G)$ for graphs 
 \[
\tilde{h}(X) := \min_{f \in C^{k-1}(X,\Z_2)} \frac{|V| \cdot |\delta_X f|}{|\delta_{K(X)} f|}.
\]
\paragraph{Complexes with Complete $(k-1)$-Skeleton.}
For $k$-complexes $X$ with a complete $(k-1)$-skeleton the two parameters $\phi(X)$ and $\tilde{h}(X)$ have a close relation, analogous to the situation for graphs.
This close relation follows from a basic observation concerning the expansion of the complete complex that was made independently in different contexts, see, e.g., \cite{Gromov:2010,MeshulamWallach:2009}:
\begin{proposition}[E.g.,~{\cite[Proposition 2.1]{MeshulamWallach:2009}}]
 \[|\delta_{K_n^k} f| \geq \tfrac{n}{k+1}|[f]| \text{ for any } f \in C^{k-1}(K_n^k,\Z_2).\]
\end{proposition}
If $X$ has a complete $(k-1)$-skeleton, $C^{k-1}(X,\Z_2)=C^{k-1}(K_n^k,\Z_2)$. As trivially $|\delta_{K_n^k} f| \leq n\cdot |[f]|$ for any $f \in C^{k-1}(K_n^k,\Z_2)$, we have the following relation between the above notions of expansion:
\[
\phi(X) \leq \tilde{h}(X) \leq (k+1)\phi(X).
\]

\paragraph{General Complexes.} For general $k$-complexes $X$ with arbitrary $(k-1)$-skeleton, this close relation does not necessarily exist.
While the trivial lower bound $\phi(X) \leq \tilde{h}(X)$ continues to hold, the existence of an upper bound is not guaranteed. 
As an example consider the complex $Z$ defined in Section~\ref{sec:Examples} (see Figure~\ref{fig:moebius}). Since $Z=K(Z)$, we have $\tilde{h}(Z) = |V|$, whereas $\phi(Z)=0$, because $H^1(Z;\Z_2)\neq 0$.

If $H^1(K(X);\Z_2) = 0$, i.e., $\phi(K(X))\neq 0$, we get analogously to the case above
\[
\phi(X) \leq \tilde{h}(X) \leq \frac{|V|}{\phi(K(X))}\cdot\phi(X).
\]
Thus, the existence (and the quality) of the upper bound depends on the expansion properties of the completion $K(X)$.

One might wonder if Theorem~\ref{thm:Theorem2} can be improved to a bound of $\lambda(X)$ by 
\[\phi'(X) = \min_{\substack{V=\amalg_{i=0}^{k-1} A_i,\, f \in C^{k-1}(X,\Z_2),\\\supp(f)\subset F(A_0,A_1,\ldots,A_{k-1})}}\frac{|\delta_X f|}{|[f]|}.\]
This, however, is impossible as is shown by the counterexamples in \cite{Steenbergen:2012}, originally constructed to show that $\phi(X)$ is not an upper bound for $\lambda(X)$.
For these complexes, $\phi(X) = \phi'(X)$, since $\phi(X)$ is attained by a cochain of this special form.

\section{Examples}\label{sec:Examples}
The following examples show that Theorem~\ref{thm:sparse} and Theorem~\ref{thm:Theorem2} can indeed give a stronger bound than Theorem~\ref{thm:Parz}. The first example shows how the three values, $\lambda(X)$, $\frac{C(X)}{n}h(X)$ and $h'(X)$ can all be different and nonzero. In the second example we have $h'(X)=0$ (and therefore $\lambda(X)=0$) but $\frac{C(X)}{n}h(X)$ nonzero. The converse is not possible since $h(X) = 0$ implies $h'(X) = 0$, but we give an instance where $\frac{C(X)}{n}h(X)$ is constant, while $h'(X)$ is linear.


Consider the complex $X$ given in Figure \ref{fig:realplane}, a triangulation of the real projective plane. It has $6$ vertices, a complete $1$-skeleton and all triangles that are visible in the figure. Let $A=\{\{1,2\},\{2,4\},\{4,5\}\}$ be the edge set depicted by bold lines and let $\1_A$ be its characteristic function, interpreted as a $\Z_2$-cochain. Then $|\delta_X \1_A| = 2$ and $|\delta_{K_6^2} \1_A| = 8$ and hence $h'(X) \leq \frac{3}{2}$. By computing $L(X)$ explicitly one can see that $\lambda(X) \approx 0.764$. We will show that $h(X)  \geq 2$. Note that, since $X$ has a complete $1$-skeleton, we furthermore have $\frac{C(X)}{|V|}h(X) = h(X)$. 
%

Consider a $3$-coloring of the vertices of $X$. In the case where there exists a color class $\{v\}$ of size one, the five neighbors of $v$ (which belong to the other two colors classes) span at least two $3$-colored triangles with $v$. In the case where all color classes have size two, one can show in a similar fashion that every such $3$-coloring (one has to distinguish between two cases) has exactly four $3$-colored triangles. Therefore $h(X) \geq \frac{|V|\cdot 2}{6} = 2$. 

\begin{figure}[ht]
   \begin{center}
    \includegraphics[width=0.225\textwidth]{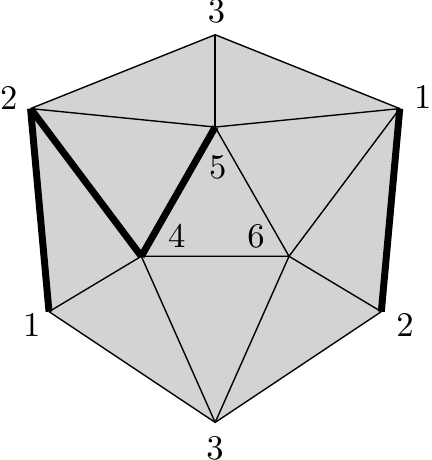}
  \end{center}
  \caption{Real Projective Plane}
\label{fig:realplane}
\end{figure}

In the second example we give a construction of a simplicial complex $Y$ where $\lambda(Y) = h'(Y) = 0$, but $\frac{C(Y)}{|V|}h(Y) = h(Y) = \frac{1}{2- \frac{6}{n}}$. Note that here $h(Y)$ is bounded by a constant, while in general a nonzero bound for $h(Y)$ can be of order $\Theta(n^{-2})$.

Let $Y=(V,E,T)$ be the following $2$-complex with complete $1$-skeleton on the vertex set $[n]$ ($n \geq 8$): We set $A = \{\{1,2\},\{2,3\},\{3,4\}\}$ and $T$ as the set of triangles that contain zero or two edges from $A$. Then by construction $|\delta_Y \1_A| = 0$ and hence $h'(Y) = 0$. Formally the set of triangles is given by 
\[ T = T^{\partial} \setminus \{\{1,2,4\}, \{1,3,4\}, \{1,2,i\}, \{2,3,i\}, \{3,4,i\}\}, \text{ where } i \in \{5,6,\dots, n\}.\]

Denote $W: = \{1,2,3,4\}$ and $\overline{W} = V \setminus W$. We determine $h(Y)$ by the following case distinction, which covers all possible cases up to permutation. Note that $C(Y) = |V|$,  since $Y$ has a complete $1$-skeleton.

\begin{enumerate}
	\item 
	$A_0$ and $A_1$ contain at most as many elements of $W$ as of $\overline{W}$ i.e.\ $|A_0 \cap W| \leq |A_0 \cap \overline{W}|$ and $|A_1 \cap W| \leq |A_1 \cap \overline{W}|$. Then $\{v_0, v_1, v_2\} \in F(A_0,A_1,A_2)$ for all $v_0 \in A_0 \cap \overline{W}$, $v_1 \in A_1 \cap \overline{W}$, $v_2 \in A_2$, and therefore
	\[\frac{|F(A_0,A_1,A_2)|}{|F^\partial(A_0,A_1,A_2)|} \geq \frac{\frac{1}{2}|A_0|\cdot \frac{1}{2}|A_1| \cdot |A_2|}{|A_0||A_1||A_2|} = \frac{1}{4}.\] 
	
	\item
	$|A_0 \cap W| \leq |A_0 \cap \overline{W}|$ but $|A_1 \cap W| > |A_1 \cap \overline{W}| \neq 0$. Since $|W| = 4$ it follows that $|A_1| \leq 7$. Hence
	\[\frac{|F(A_0,A_1,A_2)|}{|F^\partial(A_0,A_1,A_2)|} \geq \frac{\frac{1}{2}|A_0|\cdot 1 \cdot |A_2|}{|A_0||A_1||A_2|} \geq \frac{1}{14}.\]
	
	\item
	All elements of $\overline{W}$ are in $A_0$, i.e., $\overline{W} \subseteq A_0$. By checking all possible partitions, one can find that the minimum is obtained by  $A_0 = \{1,\overline{W}\}, A_1 = \{4,2\}, A_2 = \{3\}$, in which case $\frac{|F(A_0,A_1,A_2)|}{|F^\partial(A_0,A_1,A_2)|} = \frac{1}{2n-6}$.
	
	
\end{enumerate}

In our last example consider the $2$-dimensional complex $Z$ depicted in Figure \ref{fig:moebius}. It has vertex set $V = [n]$ and edge set $\{\{i,i+1 \mod n\}, \{i,i+2 \mod n\}, i \in [n]\}$. The set of triangles is $\{\{i, i+1 \text{ mod } n, i+2 \text{ mod } n\}, i \in [n]\}$. Depending on the parity of $n$, this describes either a M\"obius strip or a cylinder.
Observe that $Z=K(Z)$ and hence $h(Z) = h'(Z) = n$. 
For the partition $A_0 = \{1\}$, $A_1 = \{3\}$, $A_2 = [n]\setminus \{1,3\}$, we obtain the smallest possible value of $C(Z) = 3$.
Therefore $\frac{C(Z)}{|V|}h(Z) = \frac{3}{n}n = 3$ gives a constant bound whereas $h(Z) = h'(Z) = n$ yields a linear bound. Note that $\lambda(Z) = 0$, as $H^1(Z;\R)$ is non-trivial.

\begin{figure}[ht]
  \begin{center}
    \includegraphics[width=0.425\textwidth]{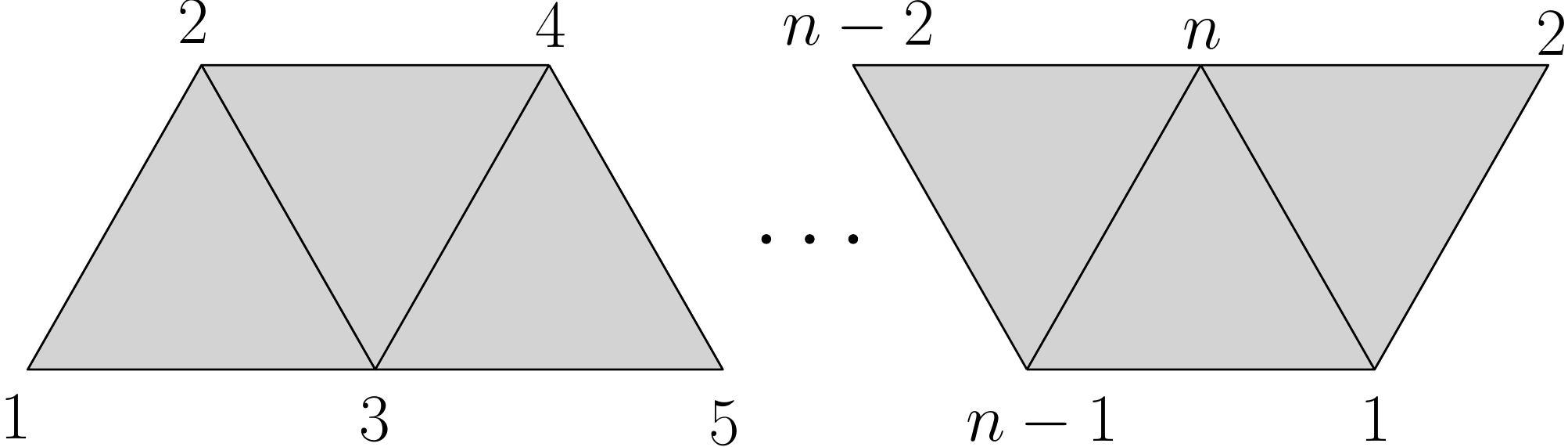}
  \end{center}
  \caption{M\"obius Strip or Cylinder}
	\label{fig:moebius}
\end{figure}

\section{The Result for Complexes with Complete Skeleton}\label{s_parzanchevski}
%

In the following part we describe the basic ideas of the proof of Theorem~\ref{thm:Parz} from \cite{Parzanchevski:2012}.
By the variational characterization of eigenvalues we know that
\begin{equation}\label{eq_ray}
\lambda(X) = \min_{f \in Z_{k-1}} \frac{\langle L_{k-1}^{\text{up}}(X)f, f \rangle}{\langle f,f \rangle}.
\end{equation}
The key idea is to find a function $f \in Z_{k-1}$ that satisfies 
\[
\frac{\langle L_{k-1}^{\text{up}}(X)f, f \rangle}{\langle f,f \rangle} = h(X).
\]
In order to define a function satisfying this equation, we fix a partition $A_0,A_1, \dots, A_k$ of $V$ which realizes the minimum in $h(X)$.
We call the $A_i$'s \emph{blocks of the partition} or shortly just \emph{blocks}.
Since the value of $\lambda(X)$ does not depend on the chosen orientation, we are free to choose an orientation depending on this partition. For reasons of simplicity we choose a linear ordering on $V$ such that for all $i<j$, $v \in A_i$, $w \in A_j$ we have $v < w$.

\label{def_fk}
Let $\sigma =\{v_0,v_1, \dots v_{k-1}\}\!\in\!X_{k-1}$, with $v_0\!<\!v_1\!<\!\dots\!<\!v_{k-1}$. Then \emph{$f \in C^{k-1}$} is defined as
\begin{equation}\label{eq:Def_f}
f(\sigma)=\begin{cases}
  (-1)^{l}|A_l| & \begin{array}{l}
  \text{if } A_l \text{ is the unique block not containing any of the } v_i,
  \end{array}\\

  0   & \begin{array}{l}\text{otherwise, i.e., if such that }v_i,v_j \in A_l.
  \end{array}
\end{cases} 
\end{equation}
%
%
%
The following two statements describing essential properties of $f$ give the proof of Theorem~\ref{thm:Parz}.
\begin{lemma} \cite{Parzanchevski:2012}
\label{lemma_complete}
Let $X$ be a $k$-dimensional simplicial complex with \emph{complete $(k-1)$-skeleton} and let $f$ be defined as above.
Then $f \in Z_{k-1}$ and 
\[
\langle f, f \rangle = |V||F^{\partial}(A_0,A_1,\dots,A_k)|=|V||A_0||A_1|\cdots|A_k|.
\]
\end{lemma}

\begin{lemma} \cite{Parzanchevski:2012}
\label{lemma_deltaf}
Let $X$ be \emph{any} $k$-dimensional simplicial complex and let $f$ be defined as above. Then
\[
\langle L_{k-1}^{\text{up}} (X) f, f \rangle =\langle \delta_{k-1} f, \delta_{k-1} f \rangle  = |V|^2 |F(A_0,A_1,\dots,A_k)|. 
\]
\end{lemma}
For the first lemma, which can be proven by a straightforward calculation, there is no trivial generalization for arbitrary simplicial complexes. The latter lemma does not require any assumptions on the $(k-1)$-skeleton and we will be able to use it for our purposes. 
For completeness we here give the proof of Lemma~\ref{lemma_deltaf}.

\begin{proof}[Proof of Lemma~\ref{lemma_deltaf}]
Let $\tau = \{v_0,v_1,\dots, v_k\} \in X_k$ with $v_0 < v_1 < \dots <v_k$. 
By definition of the coboundary operator it is enough to prove that
\[
(\delta_{k-1}f) (\tau) =\begin{cases}
  |V| & \text{if } \tau \in F(A_0,A_1,\dots,A_k),\\
  0   & \text{otherwise.}
\end{cases}
\]
First suppose that $\tau \notin F(A_0,A_1,\dots, A_k)$. If $\tau$ has three vertices in the same block $A_i$ or four vertices in two blocks, then every term $\tau \setminus \{v_i\}$ in
\[
 (\delta_{k-1}f)(\tau) = \sum_{i=0}^{k} [\tau:\tau \setminus \{v_i\}]f(\tau \setminus \{v_i\})
\]
has two vertices in the same block and hence the sum vanishes.
If we assume that $v_j,v_l$ with $v_j < v_l$ is the only pair of vertices in the same block, then by our linear ordering \mbox{$j+1=l$} and, since we have $f(v \setminus\{v_j\}) = f(v \setminus\{v_l\})$, the two non-vanishing terms
$[\tau:\tau \setminus \{v_j\}]f(\tau \setminus \{v_j\})$ and  \mbox{$[\tau:\tau \setminus \{v_{j+1}\}]f(\tau \setminus \{v_{j+1}\})$} in $(\delta_{k-1}f)(\tau)$ cancel out.

In the case where $\tau \in F(A_0,A_1,\dots, A_k)$, i.e., where we have \mbox{$v_i \in A_i$} for all $i = 0,1,\dots, k$, we get
\begin{linenomath}
\begin{equation*}
(\delta_{k-1}f)(\tau) = \sum_{i=0}^k (-1)^i f(\tau \setminus \{v_i\}) = \sum_{i=0}^k (-1)^i (-1)^{i} |A_i| 
= |V|.
\end{equation*}
\end{linenomath}
\end{proof}

\section{Our Results}
\subsection{Proof of Theorem~\ref{thm:sparse}}
In this section we give the proof of Theorem~\ref{thm:sparse}. As in Section~\ref{s_parzanchevski} we fix a partition $A_0,A_1, \dots, A_k$ of $V$ realizing the minimum in $h(X)$ and choose an orientation accordingly. We define $f$ as in \eqref{eq:Def_f}.
A key ingredient of the proof of Theorem~\ref{thm:Parz} presented in Section~\ref{s_parzanchevski} is that $f\in Z^{k-1}$. This does not hold in general. To extend the proof to arbitrary complexes, we instead study the projection of $f$ onto the space $Z^{k-1}$:
\begin{lemma} 
\label{rayleigh}
Let $f \in C^{k-1}$ be as previously defined. Then there exist unique $z \in Z_{k-1}$, $b \in B^{k-1}$ such that $f= z+b$. Furthermore
\[\lambda(X) \leq \frac{|V|^2|F(A_0,A_1,\dots,A_k)|}{\langle z,z \rangle}.\]
\end{lemma}

\begin{proof}
Since $Z_{k-1} = (B^{k-1})^{\bot}$, there exist unique co\-chains \mbox{$z \in Z_{k-1}$} and $b \in B^{k-1}$ such that $f= z+b$.
The claim follows by combining \eqref{eq_ray} with Lemma \ref{lemma_deltaf} and the fact that $\langle L_{k-1}^{\text{up}}(X) z, z \rangle = \langle L_{k-1}^{\text{up}}(X) f, f \rangle$ because
$b \in \text{ker}L_{k-1}^{\text{up}}(X)$.
\end{proof}
From now on we use $z$ and $b$ in the context of Lemma~\ref{rayleigh}.
To prove Theorem \ref{thm:sparse} we need to find a lower bound for $\langle z, z \rangle$. To the best of our knowledge, there is no way of explicitly finding $z$ by knowing $f$. We will instead make use of the fact that $b \in B^{k-1}$, i.e., there exists $g\in C^{k-2}$ such that $b = \delta_{k-2} g$, and estimate the distance of $f$ to any cochain of this form.

For the upcoming argument recall $d(\sigma)$ and $C(X)$ as defined in the introduction:
For a $(k-1)$-face $\sigma \in X_{k-1}$ we have
$
d(\sigma) = |\{\tau^{\partial} \in F^{\partial}(A_0,A_1,\dots,A_k): \sigma \subseteq \tau^{\partial}\}|
$
and
$
C(X) = \max_{\tau^{\partial} \in F^{\partial}(A_0,A_1,\dots,A_k)} \sum_{{\sigma \subseteq \tau^{\partial}}\atop \sigma \in X_{k-1}} d(\sigma). 
$

\begin{lemma}\label{lemma_combined} Let $f \in C^{k-1}$ be as previously defined and let $g \in C^{k-2}$ be arbitrary.
For $\tau^{\partial} \in F^{\partial}(A_0,A_1,\dots,A_k)$ define
\[q(\tau^{\partial},g) := \sum\nolimits_{{\sigma \subseteq \tau^{\partial}} \atop {\sigma \in X_{k-1}}}\frac{1}{d(\sigma)}(f(\sigma) - \delta_{k-2}g(\sigma))^2.\] Then
\begin{enumerate}[label=\alph*)]
 \item \label{lemma_ineq} 
$\| f - \delta_{k-2}g \|^2 \geq \sum_{\tau^{\partial} \in F^{\partial}(A_0,A_1,\dots,A_k)} q(\tau^{\partial},g).$

\item \label{prop_best}
For $\tau^{\partial}=\{v_0,v_1,\dots,v_k\} \in F^{\partial}(A_0,A_1,\dots,A_k)$ with $v_0 < v_1 < \dots < v_k$ let $d_j:= d(\tau^{\partial} \setminus \{v_j\})$. Then: 
\[
q(\tau^{\partial},g)\geq \frac{|V|^2}{\sum_{j=0}^k d_j}.
\]
\end{enumerate}
\end{lemma}
%
We first show how to use Lemma~\ref{lemma_combined} to prove Theorem \ref{thm:sparse} and then prove Lemma~\ref{lemma_combined}.

\begin{proof}[Proof of Theorem \ref{thm:sparse}]
Since $b \in B^{k-1}$ there exists some $g \in C^{k-2}$ such that $f-z = b = \delta_{k-2}g$.
By Lemma~\ref{lemma_combined} we get
\begin{linenomath}
\begin{equation*}
 \langle z,z\rangle = \|f - \delta_{k-2}g\|^2 \geq \sum_{\tau^{\partial} \in F^{\partial}(A_0,A_1,\dots,A_k)}\frac{|V|^2}{\sum_{j=0}^k d_j} \geq |F^{\partial}(A_0,A_1,\dots,A_k)|\cdot \frac{|V|^2}{C(X)},
\end{equation*}
\end{linenomath}
where the second inequality follows from the definition of $C(X)$. Combined with Lemma~\ref{rayleigh} this proves Theorem~\ref{thm:sparse}.
\end{proof}


\begin{proof}[Proof of Lemma~\ref{lemma_combined}]
\begin{enumerate}[label=\alph*)]
 \item Consider the right hand sum. 
Note that for any $\sigma \in X_{k-1}$ such that $\sigma \subseteq \tau^{\partial}$ for some $\tau^{\partial} \in F^{\partial}(A_0,A_1,\dots,A_k)$ the corresponding term
appears exactly  ${d(\sigma)}$ times by definition. 
For $\sigma \nsubseteq \tau^{\partial}$ the term does not appear at all.
The statement follows by definition of the inner product.

\item First, assume that $\tau^{\partial}= \tau \in F(A_0,A_1,\dots,A_k)$. We will see that the proof for $\tau^{\partial} \in F^{\partial}(A_0,A_1,\dots,A_k)$ works almost analogously.

Let $\tau= \{v_0,v_1,\dots,v_k\}\in X_k$ such that $v_i \in A_i$ for $i = 0,1,\dots, k$. Then 
\begin{linenomath}
\begin{align}
q(\tau,g) 
&=\sum_{i=0}^k \frac{1}{d_i} \left((-1)^{i}|A_{i}| - \delta_{k-2}g(\tau \setminus \{v_i\})\right)^2 \nonumber\\
&=\sum_{i=0}^k \frac{1}{d_i} \left(|A_{i}| - [\tau:\tau\setminus\{v_i\}]\delta_{k-2}g(\tau \setminus \{v_i\})\right)^2 \nonumber.
\end{align}
\end{linenomath}
We observe that 
$\sum_{i=0}^k [\tau:\tau \setminus \{v_i\}] \delta_{k-2} g( \tau \setminus \{v_i\}) = \delta_{k-1}(\delta_{k-2}g)(\tau) = 0$.
Using a special version of the Cauchy-Schwarz inequality, i.e., that 
\[\sum_{i=0}^k \frac{a_i^2}{b_i} \geq \frac{\left(\sum_{i=0}^k a_i\right)^2 }{\sum_{i=0}^k b_i},\]
for all $a_i \in \R, b_i \in \R^+$, we obtain
\begin{linenomath}
\begin{align}
q(\tau,g) 
&\geq \frac{\left(\sum_{i=0}^k |A_{i}| - [\tau:\tau\setminus\{v_i\}]\delta_{k-2}g(\tau \setminus \{v_i\})\right)^2}{\sum_{i=0}^k d_i} \nonumber \\
&= \frac{\left(\sum_{i=0}^k |A_{i}| \right)^2}{\sum_{i=0}^k d_i} = \frac{|V|^2}{\sum_{i=0}^k d_i}. \nonumber 
\end{align}
\end{linenomath}

In the remainder we give the proof of the statement for $\tau^{\partial}= \{v_0,v_1,\dots,v_k\} \in F^{\partial}(A_0,A_1,\dots,A_k)$ with \mbox{$v_0< v_1 <\dots < v_k$.}
Observe that the whole proof works analogously, except that we have not defined the ``oriented incidence numbers'' $[\tau^{\partial}:\sigma]$ for ``empty $k$-faces'' \mbox{$\tau^{\partial} \in F^{\partial}(A_0,A_1,\dots,A_k) \setminus F(A_0,A_1,\dots,A_k)$}.

By defining it the obvious way as $(-1)^i$ if $\sigma = \tau^{\partial} \setminus \{v_i\}$, $i = 0,1,\dots,k$ and zero otherwise, i.e., if $\sigma \nsubseteq \tau^{\partial}$,
we observe that $\delta_{k-1}\delta_{k-2}g(\tau^{\partial})=0$ and the proof works analogously.
\end{enumerate}
\end{proof}

\subsection{Proof of Theorem \ref{thm:Theorem2}}

In this section we give the proof of Theorem~\ref{thm:Theorem2}. Since we consider real as well as $\Z_2$-cohomology, we denote the real coboundary operator by $\delta^\R$, the $\Z_2$-coboundary by $\delta^{\Z_2}$. The space of $\Z_2$-cochains is denoted by $C^{k-1}(X;\Z_2)$, the space of real cochains by $C^{k-1}(X)$ instead of $C^{k-1}(X;\R)$. Also, $B^{k-1}(X)$ stands for $B^{k-1}(X;\R)$. (We now add the space $X$ to the notation, because we will consider cochains in different spaces.)

The following lemma points out a special behavior of the $\Z_2$-cochains appearing in the definition of $h'(X)$ that will be central to our argument: The size of the $\Z_2$-boundary of such a cochain agrees with the size of its real coboundary.
\begin{lemma}\label{lem:connectionL - Z2}
Let $X$ be a $k$-complex with $n$ vertices.
Let $A_0,A_1,\ldots,A_{k-1} \subset V = V(X)$ be pairwise disjoint and let $f \in C^{k-1}(X;\Z_2)$ such that $\supp(f) \subset F(A_0,A_1,\ldots,A_{k-1})$.
Choose an orientation of the simplices of $X$ by fixing a linear ordering on $V$ such that for all $i<j \in \{0,1,\dots, k-1\}$ and $v \in A_i$, $w \in A_j$ we have $v < w$.

 Then, interpreting $f$ also as an $\R$-cochain with values in $\{0,1\}$, we have
\[
\|\delta^\mathbb{R} f\|^2 = \langle L^\up_{k-1}(X) f,f \rangle = |\delta^{\Z_2} f|.
\]
Here, $\|\cdot\|$ denotes the $\ell_2$-norm and $|\cdot|$ denotes the Hamming norm. 
\end{lemma}
\begin{proof} 
Note that any $k$-face $\tau \in X_k$ can have at most two $(k-1)$-faces that are contained in $F(A_0,A_1,\ldots,A_{k-1})$, and the same holds for $\supp(f) \subset F(A_0,A_1,\ldots,A_{k-1})$.

For $\tau \in X_k$ consider 
\[
\delta^\mathbb{R} f(\tau) = \sum_{\sigma \subset \tau, \sigma \in X_{k-1}} [\tau:\sigma]f(\sigma).
\]
We distinguish several cases.
 If $\tau$ has no faces in $\supp(f)$ this sum is empty.  The sum is $\pm1$ if $t$ has exactly one face in $\supp(f)$. Otherwise $\tau$ has exactly two faces $\sigma$ and $\sigma'$ with $f(\sigma)=f(\sigma')=1$.
By our choice of orientations, we get $[\tau:\sigma] = -[\tau:\sigma']$ and hence $\delta^\mathbb{R} f(\tau)=0$.

This shows that $\langle L^\up_{k-1}(X) f,f \rangle = \|\delta^\mathbb{R} f\|^2$ equals the number of $k$-faces with exactly one face in $\supp(f)$. As we assumed that $\supp(f) \subset F(A_0,A_1,\ldots,A_{k-1})$, this is $|\delta^{\mathbb{Z}_2} f|$.
\end{proof}

Before we come to the proof of Theorem~\ref{thm:Theorem2}, we give an upper bound for the eigenvalue $\lambda(X)$. By the variational characterization of eigenvalues, $\lambda(X)$ is the minimum over all $f \in C^{k-1}(X,\R)$ of unit norm that are orthogonal to $B^{k-1}(X)$. The key observation here is that we can get rid of this orthogonality constraint.
\begin{lemma}\label{lem:Reformulation}
Let $X$ be a $k$-complex with $n$ vertices. Then
\begin{equation}\label{eq:Reformulation} 
\lambda(X) \leq \min_{\substack{f \in C^{k-1}(X),\\ f \notin B^{k-1}(X)}} \frac{n \cdot \langle L^\up_{k-1}(X)f,f \rangle}{\langle  L^\up_{k-1}(K(X)) f,f \rangle}.
\end{equation}
If $\langle  L^\up_{k-1}(K(X)) f,f \rangle =0$, we define $\frac{n \cdot \langle L^\up_{k-1}(X)f,f \rangle}{\langle  L^\up_{k-1}(K(X)) f,f \rangle} = \infty$.
For $X$ with complete $(k-1)$-skeleton \eqref{eq:Reformulation} holds with equality.
\end{lemma}

\begin{proof}
First we assume that $X$ has a complete $(k-1)$-ske\-le\-ton.
The following equality is contained implicitly in \cite{Kalai:1983} and follows from a straightforward calculation using the matrix representations of the Laplacians: 
\[
 L^\up_{k-1}(K_n^k) + L^\down_{k-1}(K_n^k) = nI. 
\]
Hence, we have for any $f \in C^{k-1}(X) = C^{k-1}(K_n^k)$:
\[
n \langle f,f \rangle = \langle L^\up_{k-1}(K_n^k) f,f \rangle + \langle  L^\down_{k-1}(K_n^k) f,f \rangle.
\]
Combining this observation with the variational characterization of eigenvalues and the fact that $L^\down_{k-1}(K_n^k) f = 0$ for $f\perp B^{k-1}(X)=B^{k-1}(K_n^k)$, we get:
\[
\lambda(X) = \min_{\substack{f \in C^{k-1}(X),\\ f \perp B^{k-1}(X)}} \frac{\langle L^\up_{k-1}(X)f,f \rangle}{\langle f,f \rangle} 
= \min_{\substack{f \in C^{k-1}(X),\\ f \perp B^{k-1}(X)}} \frac{n\cdot \langle L^\up_{k-1}(X)f,f \rangle}{\langle L^\up_{k-1}(K_n^k) f,f \rangle}.
\]
For $f \notin B^{k-1}(X)$ that is not orthogonal to $B^{k-1}(X)$, consider the projection $b$ of $f$ onto $B^{k-1}(X)$ and let $z = f-b$. Then $z\perp B^{k-1}(X)$ and it holds that 
$
\langle L^\up_{k-1}(X)z,z \rangle=\langle L^\up_{k-1}(X)f,f \rangle
$
as well as $\langle L^\up_{k-1}(K_n^k) z,z \rangle = \langle L^\up_{k-1}(K_n^k) f,f \rangle$. This shows that we can omit the orthogonality constraint.

Now, consider the general case of a $k$-complex $X$ with an arbitrary $(k-1)$-skeleton.
Let $f \in C^{k-1}(X)$. We extend $f$ to $\tilde{f} \in C^{k-1}(K_n^k)$ defined by 
 \[
 \tilde{f}(\sigma)=\begin{cases}
 f(\sigma) & \text{if } \sigma \in X,\\
 0 & \text{otherwise.}
 \end{cases}
 \]
A straightforward calculation shows that $\tilde{f} \perp B^{k-1}(K_n^k)$ if $f\perp B^{k-1}(X)$.
Hence, we can argue as above to see that for $f\perp B^{k-1}(X)$ we get
\[
n \langle f,f \rangle=n \langle \tilde{f},\tilde{f} \rangle = \langle L^\up_{k-1}(K_n^k) \tilde{f},\tilde{f} \rangle \geq \langle L^\up_{k-1}(K(X)) f,f \rangle.
\]
Thus,
\[
\lambda(X) = \min_{\substack{f \in C^{k-1}(X),\\ f \perp B^{k-1}(X)}} \frac{n\cdot \langle L^\up_{k-1}(X)f,f \rangle}{\langle L^\up_{k-1}(K_n^k) \tilde{f},\tilde{f} \rangle}
\leq \min_{\substack{f \in C^{k-1}(X),\\ f \perp B^{k-1}(X)}} \frac{n\cdot \langle L^\up_{k-1}(X)f,f \rangle}{\langle L^\up_{k-1}(K(X)) f,f \rangle}.
\]
For $f \notin B^{k-1}(X)$ that is not orthogonal to $B^{k-1}(X)$, we again consider the projection $b$ of $f$ onto $B^{k-1}(X)$. For $z = f-b$ we have $z\perp B^{k-1}(X)=B^{k-1}(K(X))$. We also get $\langle L^\up_{k-1}(X)z,z \rangle=\langle L^\up_{k-1}(X)f,f \rangle$ and $\langle L^\up_{k-1}(K(X)) z,z \rangle = \langle L^\up_{k-1}(K(X)) f,f \rangle$, which shows that also in this case we can omit the orthogonality constraint
\end{proof}

Now we can prove Theorem~\ref{thm:Theorem2}:
\begin{proof}[Proof of Theorem~\ref{thm:Theorem2}]
Fix sets $A_0,A_1,\ldots,A_{k-1} \subset V$ and a cochain $f \in C^{k-1}(X,\mathbb{Z}_2)$ with $\supp(f) \subset F(A_0,A_1,\ldots,A_{k-1})$ such that 
\[
h'(X) = n \cdot |\delta^{\Z_2}_X f|/|\delta^{\Z_2}_{K(X)} f|.
\]
If $|\delta^{\Z_2}_{K(X)} f|=0$, we have $h'(X) = \infty$ and there is nothing to show. Otherwise, we apply Lemmas~\ref{lem:connectionL - Z2} and \ref{lem:Reformulation} as follows:

Since the value of $\lambda(X)$ does not depend on the chosen orientations of the simplices of $X$, we are free to choose the orientations as in Lemma~\ref{lem:connectionL - Z2}, i.e., we fix a linear ordering on $V$ such that for all $i<j$ and $v \in A_i$, $w \in A_j$ we have $v < w$.
Then by Lemma~\ref{lem:connectionL - Z2} we get 
\[
\langle L^\up_{k-1}(X) f,f \rangle = |\delta^{\mathbb{Z}_2}_X f|
\qquad\text{and}\qquad
\langle L^\up_{k-1}(K(X)) f,f \rangle = |\delta^{\mathbb{Z}_2}_{K(X)} f|.
\]
As $|\delta^{\mathbb{Z}_2}_{K(X)} f| \neq 0$, we have $f \notin B^{k-1}(X)$ and can apply Lemma~\ref{lem:Reformulation} to obtain
\[
 \lambda (X) \leq \frac{n \cdot \langle L^\up_{k-1}(X)f,f \rangle}{\langle  L^\up_{k-1}(K(X)) f,f \rangle} = h'(X).
\]
\end{proof}

\section*{Acknowledgements}
The authors are grateful to Uli Wagner for inspiring discussions and for making us aware of the question discussed here. We would like to thank Micha{\l} Adamaszek and the referees for helpful feedback. 


\bibliographystyle{abbrv}

\end{document}